\documentclass[10pt]{amsart}
\usepackage{amssymb}
\usepackage{tikz-cd}

\title{Separation ratios of maps between Banach spaces}

\author {Christian Rosendal}
\address{Department of Mathematics\\University of Maryland\\4176 Campus Drive - William E. Kirwan Hall\\College Park, MD 20742-4015\\USA}
\email{rosendal@umd.edu}
\urladdr{sites.google.com/view/christian-rosendal/}

\newcommand{\norm}[1]{\lVert#1\rVert}
\newcommand{\Norm}[1]{\big\lVert#1\big\rVert}
\newcommand{\NORM}[1]{\Big\lVert#1\Big\rVert}
\newcommand{\triple}[1]{|\!|\!|#1|\!|\!|}

\newcommand{\TRIPLE}[1]{\Big|\!\Big|\!\Big|#1\Big|\!\Big|\!\Big|}

\newcommand{\forkindep}[1][]{\mathop{\mathop{\vcenter{\hbox{\oalign{\noalign{\kern-.3ex}
\hfil$\vert$\hfil\cr\noalign{\kern-.7ex}$\smile$\cr\noalign{\kern-.3ex}}}}}\displaylimits_{#1}}}
\newcommand{\maths}[1]{\[\begin{split}{#1}\end{split}\]}
\newcommand{\maps}[1]{\mathop{\overset{#1}\longrightarrow}}
\newcommand {\N}{\mathbb N}

\newcommand {\M}{\mathbb M}

\newcommand{\om}{\omega}
\newcommand{\eps}{\epsilon}

\newcommand{\iso}{\cong}
\newcommand{\sym}{\vartriangle}

\newcommand{\saa}{\Rightarrow}
\newcommand{\equi}{\Leftrightarrow}

\newcommand {\Del}{ \; \Big| \;}
\newcommand {\del}{ \; \big| \;}

\newcommand {\ku} {\mathcal}

\newcommand{\ov}{\overline}
\newcommand{\inv}{^{-1}}

\theoremstyle{plain}
\newtheorem{thm}{Theorem}[]

\newtheorem{lemme}[thm]{Lemma}
\newtheorem{prop} [thm] {Proposition}
\newtheorem{defi} [thm] {Definition}

\theoremstyle{definition}

\newtheorem{exa}[thm]{Example}

\newtheorem{probl}[thm]{Problem}

\definecolor{groen}{rgb}{0,0.5,.7}
\definecolor{gul}{rgb}{0.94,0.8,0}
\definecolor{blaa}{rgb}{0.16,0,0.6}
\definecolor{roed}{rgb}{1,0,0}

\makeindex

\begin{document}

\thanks{The research for this article was supported by the NSF through award DMS 2204849.}

\begin{abstract}
Under the weak assumption on a Banach space $E$ that $E\oplus E$ embeds isomorphically into $E$, we provide a characterisation of when a  Banach space $X$ coarsely embeds into $E$ via a single numerical invariant.
\end{abstract}

\maketitle

%\tableofcontents

%%%%%%%%%%%%%%%%%%%%%%%%%%%%%%%%%%%%%
%%%%%%%%%%%%%%%%%%%%%%%%%%%%%%%%%%%%%
%%%%%%%%%%%%%%%%%%%%%%%%%%%%%%%%%%%%%
%%%%%%%%%%%%%%%%%%%%%%%%%%%%%%%%%%%%%
%%%%%%%%%%%%%%%%%%%%%%%%%%%%%%%%%%%%%
%%%%%%%%%%%%%%%%%%%%%%%%%%%%%%%%%%%%%
%%%%%%%%%%%%%%%%%%%%%%%%%%%%%%%%%%%%%
%%%%%%%%%%%%%%%%%%%%%%%%%%%%%%%%%%%%%
%%%%%%%%%%%%%%%%%%%%%%%%%%%%%%%%%%%%%
%%%%%%%%%%%%%%%%%%%%%%%%%%%%%%%%%%%%%
%%%%%%%%%%%%%%%%%%%%%%%%%%%%%%%%%%%%%
%%%%%%%%%%%%%%%%%%%%%%%%%%%%%%%%%%%%%

\section{Introduction}
The concept of coarse embeddability between metric spaces can be viewed as a large scale analogue of uniform embeddability and may most easily be understood in terms of the moduli associated with a map. However, as we are exclusively concerned with Banach spaces, these moduli can further be reduced to a couple of  numerical invariants.
\begin{defi}
For a (generally  discontinuous and nonlinear) map $X\maps \phi E$ between two Banach spaces we define the {\em exact compression coefficient} $\ov \kappa(\phi)$, the {\em compression coefficient} $\kappa(\phi)$ and the {\em expansion coefficient} $\om(\phi)$ by
$$
\ov \kappa(\phi)=\sup_{t<\infty}\;\;\inf_{\norm{x-y}= t}\;\;\Norm{\phi(x)-\phi(y)},
$$
$$
\kappa(\phi)=\sup_{t<\infty}\;\;\inf_{\norm{x-y}\geqslant t}\;\;\Norm{\phi(x)-\phi(y)},
$$
and 
$$
\om(\phi)=\inf_{t>0}\;\;\sup_{\norm{x-y}\leqslant t}\;\;\Norm{\phi(x)-\phi(y)}.
$$
\end{defi}

To avoid certain trivialities, we shall tacitly assume that all Banach spaces have dimension at least $2$ and hence, in particular, that the infima and suprema above are taken over non-empty sets. 
Let us first note the obvious fact that $\om(\phi)=0$ if and only if $\phi$ is uniformly continuous. On the other hand, 
$\om(\phi)<\infty$ if and only if $\phi$ is {\em Lipschitz for large distances}, that is,
$$
\Norm{\phi(x)-\phi(y)}\leqslant K\norm{x-y}+K
$$
for some constant $K$ and all $x,y\in X$. Similarly, assumptions on $\kappa(\phi)$ correspond to known conditions on the map $\phi$. We summarise these as follows.
\begin{enumerate}
\item $\om(\phi)=0$, that is, $\phi$ is uniformly continuous,
\item $\om(\phi)<\infty$, that is, $\phi$ is Lipschitz for large distances,
\item $\kappa(\phi)=\infty$, that is, $\phi$ is {\em expanding},
\item $\kappa(\phi)>0$, that is, $\phi$ is {\em uncollapsed}.
\end{enumerate}
Note that the three coefficients above are all positive homogenous, in the sense that
$$
\kappa(\lambda\phi)=\lambda\cdot\kappa(\phi) 
$$
for all $\lambda>0$ and similarly for $\ov \kappa(\phi)$ and $\om(\phi)$. In particular, this means that the following quantity is invariant under rescaling $\phi$.

\begin{defi}\label{sep rat}
The {\em separation ratio} of  a map $X\maps \phi E$ is the quantity
$$
\ku R(\phi)=\frac{\kappa(\phi)}{\om(\phi)},
$$
where we set $\frac a\infty=\frac 0a=0$ for all $a\in [0,\infty]$ and  $\frac a0=\frac \infty a=\frac \infty 0=\infty$ for all $0<a<\infty$.
\end{defi}
Whereas $\phi$ being a uniform embedding cannot be directly expressed via the coefficients above, we note that $\phi$ is  a  {\em coarse embedding} provided that $\om(\phi)<\infty$ and  $\kappa(\phi)=\infty$, that is, if $\phi$ is Lipschitz for large distances and is expanding. 
We thus see that
$$
\ku R(\phi)=\infty
$$
if and only if $\phi$ is either uniformly continuous and uncollapsed (e.g., a uniform embedding) or if $\phi$ is a coarse embedding.

Motivated in part by the still open problem of deciding whether a Banach space $X$ coarsely embeds into a Banach space $E$ if and only it uniformly embeds, the papers \cite{Braga1, Braga2, naor-nets, equivariant, almost transitive, geomgroups} contain various constructions for producing uniform and coarse embeddings or obstructions to the same. In particular, in \cite{geomgroups} (see Theorem 1.16) we showed that, provided that $E\oplus E$ isomorphically embeds into $E$, then a uniformly continuous and uncollapsed map $X\maps\phi E$ gives rise to a simultaneously uniform and coarse embedding of $X$ into $E$. However, as shown by A. Naor \cite{naor-nets}, there are Lipschitz for large distance maps that are not even close to any uniformly continuous map. For the exclusive purpose of coarse embeddability, our main result,  Theorem \ref{main}, removes the problematic assumption of uniform continuity of $\phi$.

\begin{thm}\label{main}
Suppose $X$ and $E$ are Banach spaces so that $E\oplus E$ isomorphically embeds into $E$. Then $X$ coarsely embeds into $E$ if and only if 
$$
\sup_{\phi}\;\ku R(\phi)=\infty,
\vspace{-.15cm}
$$
where the supremum is taken over all maps $X\maps \phi E$.
\end{thm}

The proof of Theorem \ref{main} also allows us to address another issue, namely, the preservation of cotype under different forms of embeddability. For this, consider the following conditions on a map $X\maps\phi E$.
\begin{enumerate}
\item[(5)] $\ov\kappa(\phi)=\infty$, that is, $\phi$ is {\em almost expanding},
\item[(6)] $\ov\kappa(\phi)>0$, that is, $\phi$ is {\em almost uncollapsed}.
\end{enumerate}
Also, the map $\phi$ is said to be {\em solvent} provided that there are constants $R_1,R_2,\ldots$ so that
$$
R_n\leqslant \norm{x-y}\leqslant R_n+n \;\saa\; \Norm{\phi(x)-\phi(y)}\geqslant n.
$$
Provided that $\phi$ is Lipschitz for large distances, $\phi$ is solvent if and only if it is almost expanding  (see Lemma 8 \cite{equivariant}). In analogy with Definition \ref{sep rat}, we then define the {\em exact separation ratio} of $\phi$ to be
$$
\ov{\ku R}(\phi)=\frac  {\ov\kappa(\phi)}    {\om(\phi)}.
$$
As $\kappa(\phi)\leqslant \ov \kappa(\phi)$, we then have $\ku R(\phi)\leqslant \ov{\ku R}(\phi)$. Also, $\ov{\ku R}(\phi)=\infty$ exactly when $\phi$ is either uniformly continuous and almost uncollapsed or is Lipschitz for large distances and solvent.

In connection with this, B. Braga \cite{Braga2} strengthened work by M. Mendel and A. Naor \cite{mendel} to show that, if $X$ maps into a Banach space $E$ with non-trivial type by a map that is either uniformly continuous and almost uncollapsed or is Lipschitz for large distances and solvent, then 
$$
{\sf cotype}(X)\leqslant {\sf cotype}(E).
$$
The following statement therefore covers both cases of Braga's result and seemingly provides the ultimate extension in this direction.
\begin{thm}\label{main2}
Suppose $X$ and $E$ are Banach spaces so that
$$
\sup_{\phi} \ov{\ku R}(\phi)  =\infty
$$
and that $E$ has non-trivial type. Then
$$
{\sf cotype}(X)\leqslant {\sf cotype}(E).
$$
\end{thm}

Problem 7.4 in Braga's paper \cite{Braga2} asks what can be deduced about a space $X$ that admits a map $X\maps \phi E$ that is just Lipschitz for large distances and almost uncollapsed, i.e. so that $\ov{\ku R}(\phi)>0$. That is, will restrictions on the geometry of $E$ also lead to information about the geometry of $X$? In Example \ref{constant=1}, we show that this is not always so. Indeed, if $X$ is separable and $E$ is infinite-dimensional, one can always find a map $X\maps \phi E$ that is both Lipschitz for large distances and uncollapsed, i.e., so that ${\ku R}(\phi)>0$, and after renorming $E$ one can even obtain ${\ku R}(\phi)\geqslant1$. On the other hand, Theorem \ref{main2} provides a positive answer to Braga's question under the alternative assumption $\sup_{\phi} \ov{\ku R}(\phi)  =\infty$.

\vspace{.3cm}

\noindent {\bf Acknowledgements:}
I am very grateful for the extensive feedback and criticisms I got from B. Braga on a first version of this paper and for suggesting a link with cotype that led to Theorem \ref{main2}.

%%%%%%%%%%%%%%%%%%%%%%%%%%%%%%%%%%%%%%%%%%%%%%%%%%%%%%%%
%%%%%%%%%%%%%%%%%%%%%%%%%%%%%%%%%%%%%%%%%%%%%%%%%%%%%%%%
%%%%%%%%%%%%%%%%%%%%%%%%%%%%%%%%%%%%%%%%%%%%%%%%%%%%%%%%
%%%%%%%%%%%%%%%%%%%%%%%%%%%%%%%%%%%%%%%%%%%%%%%%%%%%%%%%
%%%%%%%%%%%%%%%%%%%%%%%%%%%%%%%%%%%%%%%%%%%%%%%%%%%%%%%%
%%%%%%%%%%%%%%%%%%%%%%%%%%%%%%%%%%%%%%%%%%%%%%%%%%%%%%%%
%%%%%%%%%%%%%%%%%%%%%%%%%%%%%%%%%%%%%%%%%%%%%%%%%%%%%%%%
%%%%%%%%%%%%%%%%%%%%%%%%%%%%%%%%%%%%%%%%%%%%%%%%%%%%%%%%
%%%%%%%%%%%%%%%%%%%%%%%%%%%%%%%%%%%%%%%%%%%%%%%%%%%%%%%%
%%%%%%%%%%%%%%%%%%%%%%%%%%%%%%%%%%%%%%%%%%%%%%%%%%%%%%%%

\section{Proofs}\label{sec:proofs}
Before proving our main results, let us introduce four functional moduli that lie behind the definitions of the (exact) compression and expansion coefficients.
\begin{defi}[Compression moduli]
For a (generally  discontinuous and nonlinear) map $X\maps \phi E$ between two Banach spaces we define the {\em exact compression modulus} $\ov \kappa_\phi\colon [0,\infty[\,\to [0,\infty[$ 
$$
\ov\kappa_\phi(t)=\inf\big\{\norm{\phi(x)-\phi(y)}\del \norm{x-y}= t\big\}
$$
and the {\em compression modulus} by $\ov\kappa_\phi\colon [0,\infty[\,\to [0,\infty[$ by
$$
\kappa_\phi(t)=\inf\big\{\norm{\phi(x)-\phi(y)}\del \norm{x-y}\geqslant t\big\}.
$$
\end{defi}
Thus, $\ov\kappa_\phi$ is the pointwise largest map so that 
$
\ov\kappa_\phi\big(\norm{x-x}\big)\leqslant \Norm{\phi(x)-\phi(y)}
$ 
for all $x,y\in X$, while $\kappa_\phi(t)=\inf_{r\geqslant t}\ov\kappa_\phi(r)$ is the pointwise largest monotone map satisfying the same inequality.

\begin{defi}[Expansion moduli]
For a map $X\maps \phi E$ between  Banach spaces, the  {\em exact expansion modulus} $\ov\om_\phi\colon [0,\infty[\,\to [0,\infty]$ is defined by
$$
\ov\om_\phi(t)=\sup\big\{\norm{\phi(x)-\phi(y)}\del \norm{x-y}= t\big\},
$$
and the {\em expansion modulus} $\om_\phi\colon [0,\infty[\,\to [0,\infty]$ by
$$
\om_\phi(t)=\sup\big\{\norm{\phi(x)-\phi(y)}\del \norm{x-y}\leqslant t\big\}.
$$
\end{defi}
The following are evident.
\maths{
\kappa_\phi(t)\leqslant \ov \kappa_\phi(t)\leqslant \ov \om_\phi(t)\leqslant \om_\phi(t).
}

We recall that, to avoid trivialities, all Banach spaces are assumed to have dimension at least $2$. Thus, suppose
 $X\maps \phi E$ is a map  and that $t>0$ and $x,y\in X$. Let $n\geqslant 1$ be minimal so that $\norm{x-y}\leqslant nt$, whereby $(n-1)t\leqslant \norm{x-y}$ and pick $z_0=x,z_1,z_2,\ldots, z_n=y$ so that $\norm{z_{i-1}-z_i}=t$ for $i=1,\ldots, n$. Then
$$
\Norm{\phi(x)-\phi(y)}\;\leqslant\; \sum_{i=1}^n\Norm{\phi(z_{i-1})-\phi(z_i)}\;\leqslant\; n\cdot\ov\om_\phi(t)\leqslant \frac{\ov\om_\phi(t)}{t}\norm{x-y}+\ov\om_\phi(t).
$$
In turn, this shows that
$$
\om_\phi(s)\leqslant \frac{\ov\om_\phi(t)}{t}s+\ov\om_\phi(t)
$$ 
for all $s,t>0$ and so $\limsup_{s\to 0_+}\om_\phi(s)\leqslant \inf_{t>0}\ov\om_\phi(t)$. Because  $\om_\phi$ is non-decreasing, the limit $\lim_{s\to 0_+}\om_\phi(s)=\inf_{s>0}\om_\phi(s)$ exists, whereby
\maths{
\inf_{t>0}\ov\om_\phi(t)
\leqslant \liminf_{t\to 0_+}\ov\om_\phi(t)
\leqslant \limsup_{t\to 0_+}\ov\om_\phi(t)
\leqslant\lim_{t\to 0_+}\om_\phi(t)
\leqslant \inf_{t>0}\ov\om_\phi(t).
}
All in all, we find that
$$
\om(\phi)=\inf_{t>0}\om_\phi(t)
=\lim_{t\to 0_+}\om_\phi(t)
=\lim_{t\to 0_+}\ov\om_\phi(t)=\inf_{t>0}\ov\om_\phi(t).
$$
In particular, we would obtain nothing new by introducing an {\em exact expansion coefficient} by $\ov{\om}(\phi)=\inf_{t>0}\ov\om_\phi(t)$, since this is just the expansion coefficient itself.
Furthermore, if ${\om}(\phi)<\infty$, then $\phi$ is Lipschitz for large distances, that is,
$$
\Norm{\phi(x)-\phi(y)}\leqslant K\norm{x-y}+K
$$
for some constant $K$ and all $x,y\in X$.

Next, the definition of the separation ratio may initially be difficult to parse, so let us briefly restate it more explicitly. 

\begin{lemme}
For a map $X\maps \phi E$ and a constant $K>0$, we have 
$$
\ku R(\phi)>K
$$
if and only if there are constants $\Delta,\delta, \Lambda, \lambda>0$ so that
$$
\norm{x-y}\geqslant \Delta\;\saa\; \Norm{\phi(x)-\phi(y)}\geqslant \delta,
$$
$$
\norm{x-y}\leqslant \Lambda\;\saa\; \Norm{\phi(x)-\phi(y)}\leqslant \lambda
$$
and $\frac\delta\lambda>K$.
\end{lemme}

\begin{proof}
Note that, if $\ku R(\phi)>K$, we may find $\Delta,\Lambda>0$ so that $\frac{\kappa_\phi(\Delta)}{\om_\phi(\Lambda)}>K$. Letting $\delta=\kappa_\phi(\Delta)$ and $\lambda=\om_\phi(\Lambda)$, the two implications follow. 

Conversely, if the two implications hold for some $\Delta,\delta, \Lambda, \lambda>0$ so that $\frac\delta\lambda>K$, then 
\maths{
\ku R(\phi)=\frac{\sup_{t<\infty}\kappa_\phi(t)}{\inf_{t>0}\om_\phi(t)}\geqslant \frac{\kappa_\phi(\Delta)}{\om_\phi(\Lambda)}\geqslant \frac\delta\lambda>K,
}
which verifies the lemma.
\end{proof}

\begin{proof}[Proof of Theorem \ref{main}]As noted, if $X\maps \phi E$ is a coarse embedding between arbitrary Banach spaces, then $\ku R(\phi)=\infty$, which proves one direction of implication. Also, under the stated assumption on $E$, by Theorem 1.16 \cite{geomgroups}, we have that $X$ coarsely embeds into $E$ if and only if $\ku R(\phi)=\infty$ for some map $X\maps \phi E$. So suppose instead only that $\sup_{\phi}\ku R(\phi)=\infty$. We then construct a coarse embedding $X\maps \psi E$ as follows. 

Because $E\oplus E$ embeds isomorphically into $E$, we may inductively construct two sequences $E_n, Z_n$ of closed linear subspaces of $E$ all isomorphic to $E$ so that  
$$
E_{n+1}\oplus Z_{n+1}\subseteq Z_n.
$$
Concretely, we simply begin with an isomorphic copy  $E\oplus E$ inside of $E$ and let $E_1$ and $Z_1$ be respectively the first and second summand. Again, pick an isomorphic  copy of $E\oplus E$ inside of $Z_1$ with first and second summand denoted respectively $E_2$ and $Z_2$, etc. It thus follows that
$$
E\;\supseteq\; E_1\oplus Z_1\;\supseteq\; E_1\oplus E_2\oplus Z_2\;\supseteq\;  E_1\oplus E_2\oplus E_3\oplus Z_3\;\supseteq\; \ldots
$$ 
is a decreasing sequence of closed linear subspaces of $E$. Let 
$$
V_n =E_1\oplus E_2\oplus \cdots\oplus E_n\oplus Z_n
$$
and set $V=\bigcap_{n=1}^\infty V_n$. We note that $V$ is a closed linear subspace of $E$ in which each $E_n$ is a closed subspace complemented by a bounded projection $V\maps{P_n}E_n$ so that  so that $E_m\subseteq \ker P_n$ whenever $n\neq m$. On the other hand, we have no uniform bound on the norms $\norm{P_n}$.

Fix now a sequence of isomorphisms $E\maps {T_n}E_n$ and find maps $X\maps {\theta_n} E$ with $\ku R(\theta_n)>n\,2^n\norm{P_n}\norm{T_n}\norm{T_n\inv}$. Observe that, for all $t>0$, 
$$
\kappa_{T_n\circ\theta_n}(t)\geqslant \frac{\kappa_{\theta_n}(t)}{\norm{T_n\inv}},
$$
whereas
$$
\om_{T_n\circ\theta_n}(t)\leqslant {\norm{T_n}}\cdot \om_{\theta_n}(t),
$$ 
which shows that
$$
\ku R(T_n\circ\theta_n)\geqslant \frac{\ku R(\theta_n)}{\norm{T_n}\norm{T_n\inv}}\geqslant n\,2^n\norm{P_n}.
$$
Setting $\phi_n=T_n\circ\theta_n$, we find that $\lim_n\frac{\ku R(\phi_n)}{2^n\norm{P_n}}=\infty$. The conclusion of the theorem therefore follows directly from Lemma \ref{unif saa coarse} below.
\end{proof}

\begin{lemme}\label{unif saa coarse}
Suppose $X$ and $E$ are Banach spaces and $E\maps{P_n} E$ is a sequence of bounded linear projections onto subspaces $E_n\subseteq E$ so that  $E_m\subseteq \ker P_n$ for all $m\neq n$. Assume also that there is a sequence of maps 
$$
X\maps {\phi_n} E_n
$$ 
so that 
$$
\lim_n\frac {\ku R(\phi_n)}{2^n\norm{P_n}}=\infty.
$$
Then $X$ coarsely embeds into $E$.
\end{lemme}

\begin{proof}
By composing with a translation, we may suppose that $\phi_n(0)=0$ for each $n$. Because $\lim_n\frac {\ku R(\phi_n)}{2^n\norm{P_n}}=\infty$, we may also find constants  $\Delta_n,\delta_n,\Lambda_n, \lambda_n>0$ so that
$$
\norm{x-y}\geqslant \Delta_n\;\saa\; \Norm{\phi_n(x)-\phi_n(y)}\geqslant \delta_n
$$
and 
$$
\norm{x-y}\leqslant \Lambda_n\;\saa\; \Norm{\phi_n(x)-\phi_n(y)}\leqslant \lambda_n,
$$
while 
$$
\lim_n\frac{\delta_n}{\lambda_n2^n\norm{P_n}}=\infty.
$$
For every $n$, we let
$$
\psi_n(x)=\frac{1}{\lambda_n2^n}\cdot\phi_n\Big(\tfrac{\Lambda_n}n\cdot x\Big).
$$
Then
\maths{
\norm{x-y}\leqslant n
&\;\saa\; \Norm{\tfrac{\Lambda_n}n\cdot  x-\tfrac{\Lambda_n}n\cdot  y}\leqslant \Lambda_n\\
&\;\saa\; \NORM{\phi_n\big(\tfrac{\Lambda_n}n\cdot  x\big)-\phi_n\big(\tfrac{\Lambda_n}n\cdot  y\big)}\leqslant \lambda_n\\
&\;\saa\; \norm{\psi_n(x)-\psi_n(y)}\leqslant 2^{-n}.
}
Similarly, 
\maths{
\norm{x-y}\geqslant \frac{n\Delta_n}{\Lambda_n}
&\;\saa\; \Norm{\tfrac{\Lambda_n}n\cdot  x-\tfrac{\Lambda_n}n\cdot  y}\geqslant \Delta_n\\
&\;\saa\;\NORM{\phi_n\big(\tfrac{\Lambda_n}n\cdot  x\big)-\phi_n\big(\tfrac{\Lambda_n}n\cdot  y\big)}\geqslant \delta_n\\
&\;\saa\; \Norm{\psi_n(x)-\psi_n(y)}\geqslant \frac{\delta_n}{\lambda_n2^n}.
}
In particular, if $\norm{x-y}\leqslant m$, then $\norm{x-y}\leqslant n$ for all $n\geqslant m$, whereby
$$
\sum_{n=1}^\infty\Norm{\psi_{n}(x)-\psi_{n}(y)}\leqslant \sum_{n=1}^{m-1}\Norm{\psi_{n}(x)-\psi_{n}(y)}+\sum_{n=m}^\infty2^{-n}<\infty.
$$
Also, $\psi_n(0)=0$ for all $n$, which shows that, for all $x\in X$, 
$$
\sum_{n=1}^\infty\Norm{\psi_{n}(x)}<\infty
$$
and so the series $\sum_{n=1}^\infty\psi_{n}(x)$ is  absolutely convergent in $E$. We may therefore define a map $X\maps \psi E$ by letting
$$
\psi(x)=\sum_{n=1}^\infty\psi_{n}(x).
$$

We now verify that $\psi$ is a coarse embedding of $X$ into $E$. First, let $m\geqslant 1$ be any given natural number and suppose that $x,y\in X$ satisfy $\norm{x-y}\leqslant m$. Then we may find $z_0=x, z_1, z_2, \ldots, z_m=y$ so that $\norm{z_{i-1}-z_{i}}\leqslant 1$ for all $i$ and so, in particular,  $\norm{\psi_n(z_{i-1})-\psi_n(z_i)}\leqslant 2^{-n}$ for all $n$. It thus follows that 
\maths{
\Norm{\psi(x)-\psi(y)}
&=\NORM{\sum_{n=1}^\infty\psi_{n}(x)-\sum_{n=1}^\infty\psi_{n}(y)}\\
&\leqslant \sum_{n=1}^{m-1} \Norm{\psi_n(x)-\psi_n(y)}
+ \sum_{n=m}^\infty\Norm{\psi_n(x)-\psi_n(y)}\\
&\leqslant \sum_{n=1}^{m-1}  \Norm{\psi_n(z_0)-\psi_n(z_m)}
+ \sum_{n=m}^\infty2^{-n}\\
&\leqslant \sum_{n=1}^{m-1}  \NORM{\sum_{i=1}^m\big(\psi_n(z_{i-1})-\psi_n(z_i)\big)}
+ 2^{-m+1}\\
&\leqslant \sum_{n=1}^{m-1}  \sum_{i=1}^m\Norm{\psi_n(z_{i-1})-\psi_n(z_i)}
+ 2^{-m+1}\\
&\leqslant \sum_{n=1}^{m-1}  \sum_{i=1}^m2^{-n}
+ 2^{-m+1}\\
&= \sum_{n=1}^{m-1}  m2^{-n}
+ 2^{-m+1}\\
&< m+ 2^{-m+1}.
}
In other words, for all $m$ and $x,y\in X$, we have 
$$
\norm{x-y}\leqslant m\;\saa\; \Norm{\psi(x)-\psi(y)}< m+ 2^{-m+1}.
$$
Conversely, if $m$ is any given number, find $n$ large enough so that $\frac{\delta_n}{\lambda_n2^n\norm{P_n}}\geqslant m$. Then, if $\norm{x-y}\geqslant \frac{n\Delta_n}{\Lambda_n}$, we have
\maths{
\norm{\psi(x)-\psi(y)}
&\geqslant\frac1{ \norm{P_n}}\Norm{P_{n}\psi(x)-P_{n}\psi(y)}\\
&=\frac1{ \norm{P_n}}     \Norm{\psi_n(x)-\psi_n(y)}\\
&\geqslant \frac{\delta_n}{\lambda_n2^n\norm{P_n}}\\
&\geqslant m.
}
Taken together, these two conditions show that $\psi$ is a coarse embedding.
\end{proof}

\begin{proof}[Proof of Theorem \ref{main2}]
Suppose $X$ and $E$ are Banach spaces so that 
$$
\sup_{\phi}\;\ov{\ku R}(\phi)=\infty,
\vspace{-.15cm}
$$
and $E$ have non-trivial type, i.e., ${\sf type}(E)>1$. We then note that also ${\sf type}\big(\ell_2(E)\big)={\sf type}(E)>1$ and 
${\sf cotype}\big(\ell_2(E)\big)={\sf cotype}(E)$. Thus, if we can show that $X$ maps into $\ell_2(E)$ by a map that is Lipschitz for large distances and solvent, then, by the previously mentioned result of Braga (Theorem 1.3  \cite{Braga2}), we will have that 
$$
{\sf cotype}(X)\leqslant {\sf cotype}\big(\ell_2(E)\big)={\sf cotype}(E).
$$
So fix a sequence of maps $X\maps{\phi_n}E$ so that $\ov{\ku R}(\phi_n)> n2^n$ for all $n\geqslant 1$. This means that there are $\Delta_n,\delta_n, \Lambda_n, \lambda_n>0$ so that 
$$
\norm{x-y}=\Delta_n\;\saa\; \Norm{\phi_n(x)-\phi_n(y)}\geqslant \delta_n
$$
and
$$
\norm{x-y}\leqslant \Lambda_n\;\saa\; \Norm{\phi_n(x)-\phi_n(y)}\leqslant \lambda_n
$$
and $\tfrac{\delta_n}{\lambda_n}>n2^n$. We then define $\psi_n$ by $\psi_n(x)=\tfrac{1}{2^n\lambda_n}\phi_n\big(\tfrac{\Lambda_n}nx\big)$ and note that
$$
\norm{x-y}\leqslant n\;\saa\; \Norm{\psi_n(x)-\psi_n(y)}\leqslant 2^{-n},
$$
whereas
$$
\norm{x-y}=\tfrac{n\Delta_n}{\Lambda_n}\;\saa\; \Norm{\psi_n(x)-\psi_n(y)}\geqslant n.
$$
We finally define $X\maps\psi \ell_2(E)$ by $\psi(x)=\big(\psi_1(x),\psi_2(x),\ldots\big)$ and note that $\psi$ is well-defined by the above and satisfies $\om(\psi)\leqslant \om_\psi(1)\leqslant 1$ and $\ov\kappa(\psi)\geqslant \ov\kappa_\psi\big(\tfrac{n\Delta_n}{\Lambda_n}\big)\geqslant n$ for all $n$. In other words, $\psi$ is Lipschitz for large distances and solvent.
\end{proof}

Another way to prove Theorem \ref{main2} is first to establish an analogue to Theorem \ref{main} for the quantity $\sup_{\phi}\;\ov{\ku R}(\phi)$ in place of $\sup_{\phi}\;{\ku R}(\phi)$. This is done by observing that the proof of Theorem \ref{main} above can be changed to prove the following statement.
\begin{thm}
Suppose $X$ and $E$ are Banach spaces so that $E\oplus E$ isomorphically embeds into $E$. Assume also that
$$
\sup_{\phi}\;\ov{\ku R}(\phi)=\infty,
\vspace{-.15cm}
$$
then there is a map $X\maps\phi E$ that is Lipschitz for large distances and solvent.
\end{thm}
In order to obtain Theorem \ref{main2}, one then notes that $\ell_2(E)\oplus \ell_2(E)\iso\ell_2(E)$ and so, if $\sup_{\phi}\;\ov{\ku R}(\phi)=\infty$, where the supremum is taken over all maps $X\maps\phi E$, we have a map $X\maps\psi \ell_2(E)$ that is both Lipschitz for large distances and solvent.

%%%%%%%%%%%%%%%%%%%%%%%%%%%%%%%%%%%%%%%%%%%%%%%%%%%%%%%%%%

\section{Examples}\label{sec:examples}
Theorem \ref{main} indicates that, to every pair of Banach spaces $X$ and $E$, we may associate the constant 
$$
\ku{CR}(X,E)=\sup\big\{\ku R(\phi)\del \phi\colon X\to E \text{ is a map }\big\}
$$
that under very mild assumptions on $E$  measures the extent of coarse embeddability of $X$ into $E$. The number $\ku{CR}(X,E)$ will be termed the {\em coarse embeddability ratio} of $X$ in $E$.

As the next example shows, the main interest lies in the case when $\ku{CR}(X,E)>1$, whereas $\ku{CR}(X,E)=1$ is easily obtained. 

\begin{exa}\label{constant=1}
If $X$ is separable and $E$ is a Banach space that admits an infinite {\em equilateral} set, that is, an infinite subset $A\subseteq E$ so that, for some $\delta>0$,
$$
\norm{x-y}=\delta
$$
for all distinct $x,y\in A$, then we have $\ku{CR}(X,E)\geqslant 1$. To see this, let $(Y_x)_{x\in A}$ be a partition of $X$ indexed by the set $A$ into subsets $Y_x\subseteq X$ of diameter at most $1$ and let $X\maps \phi E$ be defined by
$$
\phi(y)=x \;\equi \; x\in A\;\&\; y\in Y_x.
$$
Observe that, if $\norm{y-y'}>1$, then $y$ and $y'$ must belong to different pieces of the partition and so $\norm{\phi(y)-\phi(y')}=\delta$. On the other hand, $\norm{\phi(y)-\phi(y')}\leqslant \delta$ for all $y,y'\in X$, so we see that $
\kappa_\phi(t)=\delta$ for all $t>1$, whereas $\om_\phi(t)\leqslant \delta$ for all $t>0$. So $\ku R(\phi)\geqslant 1$.  

In particular, this reasoning applies when $E$ is one of the classical Banach spaces $\ell_p$, $c_0$, $L_p$ or even the Tsirelson space $T$.  Indeed, in these spaces, the standard unit basis $(e_n)_{n=1}^\infty$ is an infinite equilateral set (or, in the case of Tsirelson's space, $(e_n)_{n=2}^\infty$ is equilateral). Here we remark that $T^*$ is the reflexive space originally constructed and described by B. S. Tsirelson  \cite{tsirelson}, while $T$ is its $\ell_1$-asymptotic dual  whose explicit construction was given by  T. Figiel and W. B. Johnson \cite{figiel}. 

Let us also observe that, if $E$ is infinite-dimensional, then $E$ admits an equivalent renorming with respect to which it has an infinite equilateral set. Indeed, since $E$ is infinite-dimensional, it contains a normalised basic sequence $(e_n)_{n=1}^\infty$. We define a new equivalent norm $\triple\cdot$ on the closed linear space $[e_n]_{n=1}^\infty$ by letting
\maths{
\TRIPLE   {\sum_{n=1}^\infty a_ne_n}
=\sup&\Big\{    
\NORM{\sum_{n\in I} a_ne_n}+\NORM{\sum_{n\in J} a_ne_n} \;\Del \\
&I, J \text{ are intervals and $i<j$ for all $i\in I$ and $j\in J$}\Big\}.
}
As $\norm{e_n}=1$ for all $n$, we find that $\triple{e_i-e_j}=2$ for all $i<j$ and so $(e_n)_{n=1}^\infty$ is an equilateral set of the norm $\triple\cdot$. It now suffices to notice that $\triple\cdot$ extends to an equivalent norm on all of $E$. 
\end{exa}

Example \ref{constant=1} illustrates that the embeddability ratio $\ku{CR}(X,E)$ is sensitive to the specific norm on $E$, but not to the choice of norm on $X$. On the other hand,  the condition $\ku{CR}(X,E)=\infty$ only depends on the isomorphism class of $E$. Note also that, if $X$, $Y$ and $Z$ are Banach spaces so that $\ku{CR}(X,Y)=\infty$, then 
$$
\ku{CR}(X,Z)\geqslant \ku{CR}(Y,Z).
$$

An important non-embeddability result was recently established by F. Baudier, G. Lancien and T. Schlumprecht \cite{bls}, who showed that the separable Hilbert space $\ell_2$ does not coarsely embed into Tsirelson's space $T^*$. 
It is known that $T^*$ is minimal, that is, $T^*$ embeds isomorphically into all of its infinite-dimensional subspaces (see Chapter VI \cite{casazza}). Also, $T^*$ has an unconditional basis and can therefore be written as a direct sum of two infinite-dimensional subspaces. It therefore follows that $T^*\oplus T^*$ embeds isomorphically into $T^*$ and thus $E=T^*$ satisfies the assumption of Theorem \ref{main}. It follows that the coarse embeddability ratio $\ku{CR}(\ell_2,T^*)$ is finite and we now proceed to give an upper bound.

\begin{prop}
If $T^*$ denotes Tsirelson's space, then 
$$
\ku{CR}(\ell_2,T^*)\leqslant 4.
$$
\end{prop}

\begin{proof}
We rely on the analysis of \cite{bls}, which also contains additional details about the construction below. 
For the proof, assume towards a contradiction that $\ell_2\maps{\phi}E$ satisfies $\ku R(\phi)>4$. Then by pre and post-composing $\phi$ with dilations we can suppose that, for some constants $\Delta>0$ and $\delta>4$, we have
$$
\norm{x-y}\geqslant \Delta\;\saa\; \Norm{\phi(x)-\phi(y)}\geqslant \delta
$$
and
$$
\norm{x-y}\leqslant \sqrt 2\;\saa\; \Norm{\phi(x)-\phi(y)}\leqslant 1.
$$
Let $(e_n)_{n=1}^\infty$ be the standard unit vector basis for $\ell_2$  and set $\eps=\frac{\delta-4}2$. Let also $k$ be large enough so that $\sqrt{2k}\geqslant \Delta$ and let $[\N]^k$ be the collection of all $k$-element subsets of $\N$ equipped with the {\em Johnson metric},
$$
d_J(A,B)=\frac{|A\sym B|}2.
$$
Observe that $d_J$ is simply the shortest-path metric on the graph whose vertices is $[\N]^k$ and where two vertices $A$ and $B$ are connected by an edge provided that $|A\sym B|=2$. Let then $f\colon [\N]^k\to T^*$ be defined by
$$
f(A)=\phi\Big(\sum_{n\in A}e_n\Big).
$$
Observe that, if $d_J(A,B)=1$, then 
$$
\NORM{\sum_{n\in A}e_n-\sum_{n\in B}e_n}=\sqrt{|A\sym B|}=\sqrt 2
$$
and so $\norm{f(A)-f(B)}\leqslant 1$. Thus, $f$ is Lipschitz with constant $1$.

By Proposition 4.1 \cite{bls} there is an infinite subset $\M\subseteq \N$ and some $y\in T^*$ so that, for any $A\in [\N]^k$ with $A\subseteq \M$, there are vectors $y_1^A, \ldots, y_k^A\in T^*$ with $\norm{y_i^A}\leqslant 1$ so that $y, y_1^A, \ldots, y_k^A$ form a finite block basis of the standard unit vector basis for $T^*$, $k\leqslant \min{\sf supp}(y^A_1)$ and 
$$
\Norm{f(A)-(y+y_1^A+\cdots+y_k^A)}<\eps.
$$ 
In particular, for all $A,B\in [\N]^k$, $A,B\subseteq \M$, we have that 
\maths{
\norm{f(A)-f(B)}
&< \Norm{y_1^A+\cdots+y_k^A}+\Norm{y_1^B+\cdots+y_k^B}+2\eps\\
&\leqslant 2 +2 +2\eps\\
&\leqslant \delta,
}
where the second bound follows from (2.13) in \cite{bls}.
On the other hand, for any two disjoint $A, B\in [\N]^k$, we have 
$$
\NORM{\sum_{n\in A}e_n-\sum_{n\in B}e_n}=\sqrt{2k}\geqslant \Delta,
$$ 
which implies that $\norm{f(A)-f(B)}\geqslant \delta$ and thus contradicts the preceding upper bound. 
\end{proof}

The following still unsolved problem provides the main theoretical motivation for our investigations here.
\begin{probl}
Suppose $X$ and $E$ are Banach spaces. Is it true that $X$ coarsely embeds into $E$ if and only if it uniformly embeds? \end{probl}

\begin{probl}
Suppose $X$ and $E$ are Banach spaces so that $\ku{CR}(X,E) >1$. Does it follow that $\ku{CR}(X,E)=\infty$?
\end{probl}

%%%%%%%%%%%%%%%%%%%%%%%%%%%%%%%%%%%%%%%%%%%%%%%%%%%%%%%%%%%
%%%%%%%%%%%%%%%%%%%%%%%%%%%%%%%%%%%%%%%%%%%%%%%%%%%%%%%%%%%
%%%%%%%%%%%%%%%%%%%%%%%%%%%%%%%%%%%%%%%%%%%%%%%%%%%%%%%%%%%
%%%%%%%%%%%%%%%%%%%%%%%%%%%%%%%%%%%%%%%%%%%%%%%%%%%%%%%%%%%
%%%%%%%%%%%%%%%%%%%%%%%%%%%%%%%%%%%%%%%%%%%%%%%%%%%%%%%%%%%
%%%%%%%%%%%%%%%%%%%%%%%%%%%%%%%%%%%%%%%%%%%%%%%%%%%%%%%%%%%
%%%%%%%%%%%%%%%%%%%%%%%%%%%%%%%%%%%%%%%%%%%%%%%%%%%%%%%%%%%
%%%%%%%%%%%%%%%%%%%%%%%%%%%%%%%%%%%%%%%%%%%%%%%%%%%%%%%%%%%
%%%%%%%%%%%%%%%%%%%%%%%%%%%%%%%%%%%%%%%%%%%%%%%%%%%%%%%%%%%
%%%%%%%%%%%%%%%%%%%%%%%%%%%%%%%%%%%%%%%%%%%%%%%%%%%%%%%%%%%
%%%%%%%%%%%%%%%%%%%%%%%%%%%%%%%%%%%%%%%%%%%%%%%%%%%%%%%%%%%
%%%%%%%%%%%%%%%%%%%%%%%%%%%%%%%%%%%%%%%%%%%%%%%%%%%%%%%%%%%
%%%%%%%%%%%%%%%%%%%%%%%%%%%%%%%%%%%%%%%%%%%%%%%%%%%%%%%%%%%
%%%%%%%%%%%%%%%%%%%%%%%%%%%%%%%%%%%%%%%%%%%%%%%%%%%%%%%%%%%
%%%%%%%%%%%%%%%%%%%%%%%%%%%%%%%%%%%%%%%%%%%%%%%%%%%%%%%%%%%
%%%%%%%%%%%%%%%%%%%%%%%%%%%%%%%%%%%%%%%%%%%%%%%%%%%%%%%%%%%
%%%%%%%%%%%%%%%%%%%%%%%%%%%%%%%%%%%%%%%%%%%%%%%%%%%%%%%%%%%
%%%%%%%%%%%%%%%%%%%%%%%%%%%%%%%%%%%%%%%%%%%%%%%%%%%%%%%%%%%

\end{document}